\documentclass[11pt]{amsart}
\usepackage[all,2cell,ps]{xy}
\usepackage{amssymb}
\usepackage{amsmath}
\usepackage{amsfonts}
\usepackage[margin=1in]{geometry}
\usepackage{color}
\usepackage[notcite,notref]{}
\usepackage{url}

\def\latex/{{\protect\LaTeX}}
\def\latexe/{{\protect\LaTeXe}}
\def\amslatex/{{\protect\AmS-\protect\LaTeX}}
\def\tex/{{\protect\TeX}}
\def\amstex/{{\protect\AmS-\protect\TeX}}
\def\bibtex/{{Bib\protect\TeX}}
\def\makeindx/{\textit{MakeIndex}}

\usepackage{amsmath}
\usepackage{amsthm}
\usepackage{amssymb}
\usepackage{amscd}
\usepackage{amsxtra}     
\usepackage{epsfig}
\usepackage{verbatim}
\usepackage[all]{xypic}

\theoremstyle{plain} 

\newtheorem{thm}{Theorem}[section]
\newtheorem{lem}[thm]{Lemma}
\newtheorem{prop}[thm]{Proposition}

\newtheorem{cor}[thm]{Corollary}

\theoremstyle{definition}

\newtheorem{eg}[thm]{Example}
\newtheorem{conj}[thm]{Conjecture}

\newtheorem{rmk}[thm]{Remark}

\newcommand{\ZZ}{\mathbb{Z}}

\newcommand{\tensor}{\otimes}

 \DeclareMathOperator{\Tor}{Tor}
\DeclareMathOperator{\Ext}{Ext}
\DeclareMathOperator{\Hom}{Hom}
\DeclareMathOperator{\CI}{\textnormal{CI-dim}}

 \DeclareMathOperator{\Spec}{Spec}

 \DeclareMathOperator{\MCM}{MCM}

 \DeclareMathOperator{\pd}{pd}
 
 \DeclareMathOperator{\cx}{cx}

 \DeclareMathOperator{\depth}{depth}
 
 \DeclareMathOperator{\Cl}{Cl}

\DeclareMathOperator{\edim}{edim}
\DeclareMathOperator{\ds}{\displaystyle}
 \DeclareMathOperator{\syz}{syz}

\newcommand{\ses}[3]{0 \to {#1} \to {#2} \to {#3} \to 0}

\begin{document}

\title{Necessary conditions for  the depth formula over Cohen-Macaulay local rings}
\author{Olgur Celikbas and Hailong Dao}

\address{Department of Mathematics, University of Kansas, 405 Snow Hall, 1460 Jayhawk
Blvd, Lawrence, KS 66045-7523, USA}
\email{ocelikbas@math.ku.edu}

\address{Department of Mathematics, University of Kansas, 405 Snow Hall, 1460 Jayhawk
Blvd, Lawrence, KS 66045-7523, USA}
\email{hdao@math.ku.edu}

\thanks{The second author is partially supported by NSF grant DMS 0834050}



\maketitle

\begin{abstract} Let $R$ be a Cohen-Macaulay local ring and let $M$ and $N$ be finitely generated $R$-modules. In this paper we investigate some of the necessary conditions for the depth formula $\depth(M)+\depth(N)=\depth(R)+\depth(M\otimes_{R}N)$ to hold. We show that, under certain conditions, $M$ and $N$ satisfy the depth formula if and only if $\Tor_{i}^{R}(M,N)=0$ for all $i\geq 1$. We also examine the relationship between the depth of $M\otimes_RN$ and the vanishing of $\Ext$ modules with various applications. One of them extends partially a result by Auslander on even dimensional regular local rings to complete intersections.  In another application, we show that there is no nontrivial semidualizing module over Veronese subrings of the formal power series ring $k[[X_{1}, \dots, X_{n}]]$ over a field $k$.
\end{abstract}

\section{Introduction}

Let $(R, \mathfrak m, k)$ be a local ring (a commutative Noetherian ring with unique maximal ideal $\mathfrak m$) and let $M$ and $N$ be finitely generated nonzero $R$-modules. We say the pair $(M,N)$ satisfies the {\it depth formula} provided:
$$\depth(M)+\depth(N)=\depth(R)+\depth(M\otimes_{R}N)$$

This useful formula is not true in general, as one can see by taking $R$ to  have depth at least $1$ and $M=N=k$.

The depth formula was first studied by Auslander \cite{Au} in 1961 for finitely generated modules over regular local rings. More precisely, if $R$ is a local ring, Auslander proved that $M$ and $N$ satisfy the depth formula provided $M$ has finite projective dimension and $\Tor_i^R(M,N)=0$ for all $i\geq 1$ \cite[1.2]{Au}. Three decades later Huneke and Wiegand \cite[2.5]{HW1} proved that the depth formula holds for $M$ and $N$ over {\it complete intersection} rings $R$ provided $\Tor_i^R(M,N)=0$ for all $i\geq 1$, even if $M$ does not have finite projective dimension. Recall that $R$ is said to be a \textit{complete intersection} if the defining ideal of some (equivalently every) Cohen presentation of the $\mathfrak m$-adic completion $\widehat{R}$ of $R$ can be generated by a regular sequence. If $R$ is such a ring, then $\widehat{R}$ has the form $Q/(\underline{f})$, where $\underline{f}$ is a regular sequence of $Q$ and $Q$ is a ring of formal power series over the field $k$, or over a complete discrete valuation ring with residue field $k$. 

There are plenty of sufficient conditions in the literature for $M$ and $N$ to satisfy the depth formula, cf., for example, \cite{BerJor}, \cite{IC}, \cite{CJ}, \cite{HW1}, \cite{Jo1} and \cite{Mi}. A common ingredient of those conditions is the vanishing of $\Tor_i^R(M,N)$ for all $i\geq 1$. In particular the following theorem, proved independently by Araya-Yoshino \cite[2.5]{ArY} and Iyengar \cite[4.3]{I}, shows that the vanishing of $\Tor_{i}^{R}(M,N)$ for modules of finite complete intersection dimension is sufficient for the depth formula (cf. Section \ref{Tor} for the definition of finite complete intersection dimension.)   

\begin{thm}\label{Choi}(Araya-Yoshino \cite{ArY}, Iyengar \cite{I}) Let $R$ be a local ring and let $M$ and $N$ be finitely generated $R$-modules. Assume that $M$ has finite complete intersection dimension (e.g., $R$ is a complete intersection.) If $\Tor_{i}^{R}(M,N)=0$ for all $i\geq 1$, then $(M,N)$ satisfies the depth formula. 
\end{thm}

We note that Iyengar's result \cite[Section 4]{I} that concerns the depth formula is more general than the one stated in Theorem \ref{Choi}; it establishes the validity of the (derived) depth formula for certain complexes of modules (cf. also \cite{CJ} and \cite{Foxby}).

The purpose of this note is to understand {\it necessary} conditions for the depth formula. From the known results, one obvious candidate is  the vanishing of $\Tor_{i}^{R}(M,N)$ for all $i\geq 1$. In general, one cannot hope for such a phenomenon; the depth formula is satisfied by any pair of finitely generated modules $(M,N)$ such that $\depth(M) =\depth(R)$ and $N$ has finite length. It is therefore somewhat surprising that a partial converse of Theorem \ref{Choi} can be obtained for a special class of rings. Here is one of the main corollaries of our results which we will prove in Section \ref{Tor}. Recall that the embedding dimension of $R$, denoted by $\edim(R)$, is the minimal number of generators of $\mathfrak m$.
 
\begin{thm} \label{main1} Let $R$ be a Cohen-Macaulay local ring and let $M$ and $N$ be non-zero finitely generated $R$-modules. Set $e= \edim(R)- \depth(R)$ and assume:
\begin{enumerate}
\item $M$ has finite complete intersection dimension.
\item $R_{p}$ is regular for each prime ideal $p$ of $R$ of height at most $e$. 
\item $\Tor^{R}_{i}(M,N)=0$ for all $i=1, \dots,  \depth(R)-\depth(M\otimes_{R}N)$.
\end{enumerate}
Then $(M,N)$ satisfies the depth formula if and only if $\Tor^{R}_{i}(M,N)=0$ for all $i\geq 1$.
\end{thm}

In particular, we have:

\begin{cor} Let $R$ be a Cohen-Macaulay local ring and let $M$ and $N$ be non-zero finitely generated $R$-modules. Set $e=\edim(R)- \depth(R)$ and assume:
\begin{enumerate}
\item $M$ has finite complete intersection dimension.
\item $R_{p}$ is regular for each prime ideal $p$ of $R$ of height at most $e$. 
\item $\depth(M) =\depth(N) =\depth(R)$.
\end{enumerate}
Then $(M,N)$ satisfies the depth formula if and only if $\Tor^{R}_{i}(M,N)=0$ for all $i\geq 1$.
\end{cor}

In Section \ref{Ext} we investigate similar conditions on the depth of $M\tensor_RN$ that force the vanishing of $\Ext$ modules.  Such conditions turn out to be quite useful for various applications. Our main result in this section is Theorem \ref{t1}. A consequence of this theorem, Corollary \ref{cor1}, implies:

\begin{cor} Let $(R, \mathfrak m)$ be a $d$-dimensional local Cohen-Macaulay ring ($d\geq 1$) with a canonical module $\omega_R$ and let $M$ and $N$ be maximal Cohen-Macaulay $R$-modules. Assume $R$ has an isolated singularity. Then  $(M,N)$ satisfies the depth formula if and only if $\Ext^i_R(M, \Hom_R(N,\omega_R))=0$ for all $i=1,\dots ,d$. 
\end{cor}

Recall that a local ring $(R, \mathfrak m)$ is said to be have an \emph{isolated singularity} provided $R_{p}$ is regular for all prime ideals $p$ of $R$ with $p \neq \mathfrak{m}$. A nonzero finitely generated module $M$ is called a \emph{maximal Cohen-Macaulay} module in case $\depth(M)=\dim(R)$. 

We exploit the main result of Section \ref{Ext} in several directions. For example, Corollary \ref{c6} is a partial extension of a Theorem of Auslander \cite[3.7]{Au} that concerns the torsion-freeness of $M\tensor M^*$ ($M^{\ast}=\Hom(M,R)$) over regular local rings $R$:

\begin{cor} \label{t2} Let $R$ be an even dimensional complete intersection that has an isolated singularity and let $M$ be a maximal Cohen-Macaulay $R$-module. If $\depth(M\tensor M^*)>0$, then $M$ is free.
\end{cor}

We also give a short proof of a result of Huneke and Leuschke \cite{HL}; this is a special case of the Auslander-Reiten conjecture \cite{AuRe} for Gorenstein rings (cf. also \cite{Ar}). 

In Proposition \ref{goodDepth} and Example \ref{exDepth}, we discuss a fairly general method to construct nonfree finitely generated modules $M$ and $N$ over certain Cohen-Macaulay normal local domains $R$ such that the tensor product $M\tensor_RN$ has high depth. Our example shows that, in contrast to the well-studied cases over regular or complete intersection rings, good depths of tensor products is a less restrictive phenomenon in general. 

Finally, we apply Theorem \ref{t1} to show that over Veronese subrings of power series rings, the only semi-dualizing modules are free or dualizing. The study of semi-dualizing modules has attracted some attention lately, and our results contribute a new class of rings whose semi-dualizing modules are completely understood.

\section{On the converse of the depth formula}\label{Tor}

This section is devoted to the connection between the depth formula and the vanishing of $\Tor$ modules. We start by recording the following observation:

\begin{lem} \label{l1} Let $R$ be a local ring and let $M$ and $N$ be nonzero finitely generated $R$-modules. Assume $(M,N)$ satisfies the depth formula and that $\depth(R)>\depth(M)$. Set $M'=\syz^{R}_1(M)$. If $\Tor^{R}_{1}(M,N)=0$, then $\depth(M'\otimes_{R}N)=\depth(M\otimes_{R}N)+1$ and hence $(M',N)$ satisfies the depth formula.
\end{lem}

\begin{proof} As $\Tor^{R}_{1}(M,N)=0$, one has the following exact sequence
$$\ses {M'\tensor_RN}{F\tensor_RN }{M\tensor_RN},$$
for some free $R$-module $F$. Notice that the assumptions imply $\depth(F\tensor_RN)>\depth(M\tensor_RN)$. Hence counting the depths of the modules in the exact sequence above gives the desired result. 
\end{proof}

We now recall some definitions needed for the rest of the paper.

If $\textbf{F}:\ldots \rightarrow F_{2}\rightarrow F_{1} \rightarrow F_{0}  \rightarrow 0$ is a minimal free resolution of $M$ over $R$, then the rank of $F_{n}$, that is, the integer $\dim_{k}(\Ext^{n}_{R}(M,k))$, is the $n$th \textit{Betti} number $\beta^{R}_{n}(M)$ of $M$. This integer is well-defined for all $n$ since minimal free resolutions over $R$ are unique up to isomorphism. 

$M$ has \emph{complexity} $r$, written as $\cx_{R}(M)=r$, provided $r$ is the least nonnegative integer for which there exists a real number $\gamma$ such that $\beta^{R}_{n}(M)\leq \gamma \cdot n^{r-1}$ for all $n\gg 0$ \cite[3.1]{Av1}. If there are no such $r$ and $\gamma$, then one sets $\cx_{R}(M)=\infty$.

The notion of complexity, which is a homological characteristic of modules, was first introduced by Alperin in \cite{Alp} to study minimal projective resolutions of modules
over group algebras. It was then brought into local algebra by Avramov \cite{Av1}. The complexity of $M$ measures how the Betti sequence $\beta^{R}_{0}(M), \beta^{R}_{1}(M), \dots$ behaves with respect to polynomial growth. In general complexity may be infinite; for example, if $R=k[X,Y]/(X^{2},XY,Y^{2})$, then $\cx_{R}(M)\in \{0,\infty \}$ \cite[4.2.2]{Av2}. It follows from the definition of  complexity that $M$ has finite projective dimension if and only if $\cx_{R}(M)=0$, and has bounded Betti numbers if and only if $\cx_{R}(M)\leq 1$.

A \textit{quasi-deformation} of $R$ \cite{AGP} is a diagram $R \rightarrow S \twoheadleftarrow P$ of local homomorphisms, where $R\rightarrow S$ is flat and $S\twoheadleftarrow P$ is surjective with kernel generated by a regular sequence of $P$ contained in the maximal ideal of $P$. $M$ is said to have finite \textit{complete intersection dimension}, denoted by $\CI_{R}(M)<\infty$, if there
exists a quasi-deformation $R \rightarrow S \twoheadleftarrow P$ such that $\pd_{P}(M\tensor_{R}S)<\infty$. It follows from the definition that modules of finite projective dimension and modules over complete intersection rings have finite complete intersection dimension. There are also local rings $R$ that are not complete intersections, and finitely generated $R$-modules that do not have finite projective dimension but have finite complete intersection dimension (cf. for example \cite[\textnormal{Chapter 4}]{AGP}).

A result of Avramov, Gasharov and Peeva shows that finite complete intersection dimension implies finite complexity; more precisely, if $\CI_{R}(M)<\infty$, then $\cx_{R}(M) \leq \edim(R)-\depth(R)$ \cite[5.6]{AGP}. In particular, if $R$ is a complete intersection, then $\cx_{R}(M)$ cannot exceed $\edim(R)-\dim(R)$, namely the \emph{codimension} of $R$ (cf. also \cite{Gu}).

Assume that the natural map $M \rightarrow M^{\ast\ast}$ is injective. Let $\{f_{1},f_{2},\dots, f_{m}\}$ be a minimal generating set for $M^{\ast}$ and let $\displaystyle{\delta: R^{(m)} \twoheadrightarrow M^{\ast}}$ be defined by $\delta(e_{i})=f_{i}$ for $i=1,2,\dots, m$ where $\{e_{1},e_{2},\dots, e_{m}\}$ is the standard basis for $R^{(m)}$. Then, composing the natural map $M\hookrightarrow M^{\ast\ast}$ with $\delta^{\ast}$, we obtain a short exact sequence
$$\;\;\textnormal{(PF)}\;\;\; 0 \rightarrow M \stackrel{u}{\rightarrow} R^{(m)} \rightarrow M_{1} \rightarrow 0$$ where $u(x)=(f_{1}(x),f_{2}(x),\dots, f_{m}(x))$ for all $x\in M$.
Any module $M_{1}$ obtained in this way is referred to as a \emph{pushforward} of $M$ \cite[Lemma 3.4 and page 49]{EG} (cf. also \cite{HJW}). We should note that such a construction is unique, up to a non-canonical isomorphism (cf. \textnormal{page} 62 of \cite{EG}). 

Throughout the rest of the paper $X^{n}(R)$ denotes the set $\{p\in \text{Spec}(R): \text{depth}(R_{p}) \leq n \}$.

\begin{thm} \label{p1} Let $R$ be a  Cohen-Macaulay local ring and let $M$ and $N$ be nonzero finitely generated $R$-modules. Let $w$ be a nonnegative integer such that $w\leq r$ where $r=\cx_{R}(M)$. Assume:
\begin{enumerate}
\item $\CI_{R}(M)<\infty$.
\item $\Tor_{i}^{R}(M,N)_{p}=0$ for all $i\gg 0$ and for all $p \in X^{r-w}(R)$.
\item $\Tor^{R}_{i}(M,N)=0$ for all $i=1, \dots,  \depth(R)-\depth(M\otimes_{R}N)+w$.
\end{enumerate}
Then $(M,N)$ satisfies the depth formula if and only if $\Tor^{R}_{i}(M,N)=0$ for all $i\geq 1$.
\end{thm}

\begin{proof} By Theorem \ref{Choi} it suffices to prove the case where $(M,N)$ satisfies the depth formula. Assume $(M,N)$ satisfies the depth formula and set $n(M,N)= \depth(R)-\depth(M\otimes_{R}N)$. We shall proceed by induction on $n(M,N)$. Since any syzygy of $M$ has finite complete intersection dimension \cite[1.9]{AGP}, in view of Lemma \ref{l1} and the induction hypothesis, it is enough to prove the case where $n(M,N)=0$. Assume now $n(M,N)=0$. Then $M$, $N$ and $M\otimes_{R}N$ are maximal Cohen-Macaulay. We may assume $r>0$; otherwise $M$ will be free by the Auslander-Buchsbaum formula. Since $\CI_{R}(M)=0$ \cite[1.4]{AGP} and $\ds{\CI_{R_{p}}(M_{p})\leq \CI_{R}(M)}$ for each prime ideal $p$ of $R$ \cite[1.6]{AGP}, we have, by (2) and \cite[4.2]{ArY}, that $\Tor_{i}^{R}(M,N)_{p}=0$ for all $i\geq 1$ and for all $p \in X^{r-w}(R)$. This shows, since $r-w\geq 0$, we may assume that $\dim(R)>0$.

As $\CI_{R}(M)=0$, by the pushforward construction (cf. also \cite[Proposition 11]{Mas}), there are exact sequences
$$(\ref{p1}.1)\;\;\; 0 \rightarrow M_{j-1} \rightarrow F_{j} \rightarrow M_{j} \rightarrow 0 $$
where $M_{0}=M$, $\CI_{R}(M_{j})=0$ and $F_{j}$ is a finitely generated free $R$-module, for each positive integer $j$. Assume now $j=1$. Then tensoring $(\ref{p1}.1)$ with $N$ and noting that $\Tor_1^R(M_1,N)$ is not supported in  $X^{r-w}(R)$, we conclude that $\Tor_1^R(M_1,N)=0$. If $j\geq 2$, continuing in a similar fashion, we see $\Tor^{R}_{1}(M_{r-w+1},N)=\dots=\Tor^{R}_{r-w+1}(M_{r-w+1},N)=0$; this follows from the fact that $M_{i}\otimes_{R}N$ is torsion-free for all $i=0,1, \dots, r-w$ (cf. the proof of \cite[2.1]{HJW}). Hence, by (3), we have $\Tor^{R}_{1}(M_{r-w+1},N)=\dots=\Tor^{R}_{r+1}(M_{r-w+1},N)=0$. It now follows from \cite[2.6]{Jor} that $\Tor^{R}_{i}(M_{r-w+1},N)=0$ for all $i\geq 1$. Thus $(\ref{p1}.1)$ shows that $\Tor^{R}_{i}(M,N)=0$ for all $i\geq 1$.
\end{proof}

Theorem \ref{main1} advertised in the introduction now follows rather easily:

\begin{proof}(of Theorem \ref{main1}) We know, by \cite[5.6]{AGP}, that $\cx_R(M)\leq e$. Therefore Proposition \ref{p1} yields the desired conclusion.
\end{proof}

\begin{cor} \label{corintro} Let $R$ be a complete intersection of codimension $c$ and let $M$ and $N$ be nonzero finitely generated $R$-modules. Assume:
\begin{enumerate}
\item $R_{p}$ is regular for each $p\in X^{c}(R)$.
\item $\Tor^{R}_{i}(M,N)=0$ for all $i=1, \dots,  \depth(R)-\depth(M\otimes_{R}N)$.
\end{enumerate}
Then $(M,N)$ satisfies the depth formula if and only if $\Tor^{R}_{i}(M,N)=0$ for all $i\geq 1$.
\end{cor}

It follows from Corollary \ref{corintro} and \cite[4.7 \textnormal{and} 6.1]{AvBu} that:

\begin{cor} Let $R$ be a complete intersection of codimension $c$ and let $M$ and $N$ be maximal Cohen-Macaulay $R$-modules. Assume $R_{p}$ is regular for each $p\in X^{c}(R)$. Then the following are equivalent:
\begin{enumerate}
\item $M\otimes_{R}N$ is maximal Cohen-Macaulay.
\item $\Tor^{R}_{i}(M,N)=0$ for all $i\geq 1$. 
\item $\Ext^{i}_{R}(M,N)=0$ for all $i\geq 1$.
\item $\Ext^{i}_{R}(N,M)=0$ for all $i\geq 1$.
\end{enumerate}
\end{cor}

\section{Depth of $M\otimes_{R}N$ and the vanishing of $\Ext_{R}^{i}(M,N)$}\label{Ext}

In this section we investigate the connection between the depth of $M\otimes_{R}N$ and the vanishing of certain $\Ext$ modules. We illustrate this connection by giving a number of applications, including another proof of a special case of the Auslander-Reiten conjecture, which was first proved by Huneke and Leuschke in \cite{HL}. We also extend a result of Auslander (on the torsion-freeness of $M\otimes_{R}M^{\ast}$) to hypersurface singularities and discuss its possible extension to complete intersections. We finish this section by showing a method to construct non-trivial examples of modules $M$ and $N$ such that $M\tensor_RN$ has high depth. 

If $R$ has a canonical module $\omega_{R}$, we set $M^{\vee}=\Hom_R(M,\omega_{R})$. $M$ is said to be locally free on $U_{R}$ provided $M_{p}$ is a free $R_{p}$-module for all $p\in X^{d-1}(R)$ where $d=\dim(R)$ (Recall that $X^{n}(R)=\{p\in \text{Spec}(R): \text{depth}(R_{p}) \leq n \}$).

Now we follow the proof of \cite[3.10]{Yo} and record a useful lemma:

\begin{lem}\label{ExtTor}
Let $R$ be a $d$-dimensional local Cohen-Macaulay ring with a canonical module $\omega_{R}$ and let $M$ and $N$ be finitely generated  $R$-modules. Assume $M$ is locally free on $U_R$ and $N$ is maximal Cohen-Macaulay. Then the following isomorphism holds for all positive integers $i$:
$$ \Ext_R^{d+i}(M,N^{\vee}) \cong \Ext_{R}^d(\Tor_i^R(M,N),\omega_{R})$$
\end{lem}

\begin{proof} By \cite[10.62]{Roit} there is a third quadrant spectral sequence:
$$\Ext^p_{R}(\Tor_q^R(M,N),\omega_{R}) \underset{p}{\Longrightarrow} \Ext_R^n(M,N^{\vee}) $$
Since $M$ is locally free on $U_R$, $\Tor_q^R(M,N)$ has finite length for all $q>0$. Therefore, unless $p=d$, $\Ext^p_{R}(\Tor_q^R(M,N),\omega_{R})=0$. It follows that the spectral sequence considered collapses and hence gives the desired isomorphism.
\end{proof}

Before we prove the main result of this section, we recall some relevant definitions and facts. An $R$-module $M$ is said to satisfy \emph{Serre's condition} $(S_{n})$ if $\depth_{R_{p}}(M_{p})\geq \text{min}\left\{n, \dim(R_{p})\right\}$ for all $p\in \Spec(R)$ \cite{EG}. If $R$ is Cohen-Macaulay, then $M$ satisfies $(S_{n})$ if and only if every $R$-regular sequence $x_{1},x_{2},\dots, x_{k}$, with $k\leq n$, is also an $M$-regular sequence \cite{Sam}. In particular, if $R$ is Cohen-Macaulay, then $M$ satisfies $(S_{1})$ if and only if it is torsion-free. Moreover, if $R$ is Gorenstein, then $M$ satisfies $(S_{2})$ if and only if it is reflexive, that is, the natural map $M \rightarrow M^{\ast\ast}$ is bijective (cf. \cite[3.6]{EG}). More generally, over a local Gorenstein ring, $M$ satisfies $(S_{n})$ if and only if it is an $n$th syzygy (cf.  \cite{Mas}).

We now improve a theorem of Huneke and Jorgensen \cite{HJ}; A special case of Theorem \ref{t1}, namely the case where $R$ is Gorenstein and $n=\dim(R)$, was proved in \cite[5.9]{HJ} by different techniques (cf. also \cite[4.6]{HW1}). 

\begin{thm} \label{t1} Let $R$ be a $d$-dimensional local Cohen-Macaulay ring with a canonical module $\omega_R$ and let $M$ and $N$ be finitely generated $R$-modules. Assume;
\begin{enumerate}
\item There exists an integer $n$ such that $1\leq n \leq \depth(M)$.
\item $M$ is locally free on $U_R$.
\item  $N$ is maximal Cohen-Macaulay.
\end{enumerate}
Then $\depth(M\otimes_{R}N)\geq n$ if and only if $\Ext^i_R(M,N^{\vee})=0$ for all $i=d-n+1,\dots, d-1,d$.
\end{thm}

\begin{proof} Note that $M$ is a torsion-free $R$-module; it is locally free on $U_R$ and has positive depth. Hence it follows from \cite[1.4.1(a)]{BH} (cf. also the proof of \cite[3.5]{EG}) that the natural map $M \rightarrow M^{\ast\ast}$ is injective. Therefore, by the pushforward construction (cf. section \ref{Tor}), there are exact sequences
$$(\ref{t1}.1)\;\;\; 0 \rightarrow M_{j-1} \rightarrow F_{j} \rightarrow M_{j} \rightarrow 0 $$
where $M_{0}=M$, $F_{j}$ is a finitely generated free $R$-module and $\depth(M_j) \geq \depth(M)-j$ for all $j=1,\dots,\depth(M)$. Furthermore it is clear from the construction that each $M_j$ is also locally free on $U_R$.

We now proceed by induction on $n$. As $N$ is maximal Cohen-Macaulay and $M$ is locally free on $U_R$, $\depth(M\otimes_{R}N)\geq n$ if and only if $M\otimes_{R}N$ satisfies Serre's condition $(S_{n})$. First assume $n=1$. We will prove that $M\otimes_{R}N$ is torsion-free if and only if $\Ext^d_R(M,N^{\vee})=0$. Consider the short exact sequence: 
$$(\ref{t1}.2)\;\;\;  \ses{M}{F_1}{M_1}$$
Tensoring (\ref{t1}.2) with $N$, we see that $M\tensor_RN$ is torsion-free if and only if $\Tor_1^R(M_{1},N)=0$. Therefore, by Lemma \ref{ExtTor} and \cite[3.5.8]{BH}, $M\otimes_{R}N$ is torsion-free if and only if $\Ext_R^{d+1}(M_{1}, N^{\vee})\cong \Ext_R^{d}(M, N^{\vee})=0$. This proves, in particular, the case where $d=1$. Hence we may assume $d\geq 2$ for the rest of the proof.

Now assume $n>1$. We claim that the following are equivalent: 
$$\text{(i)} \depth(M\tensor_RN)\geq n \text{, (ii)} \depth (M\tensor_RN)\geq 1 \text{ and} \depth (M_1\tensor_RN)\geq n-1.$$ 
Note that, if either (i) or (ii) holds, then $M\otimes_{R}N$ is torsion-free and hence $\Tor_1^R(M_{1},N)=0$. Thus we have the following exact sequence:
$$(\ref{t1}.3)\;\;\; \ses{M\tensor_RN}{F_1\tensor_RN}{M_1\tensor_RN} $$ 
Now it is clear that (i) implies (ii) since the module $F_1\tensor_RN$ in (\ref{t1}.3) is maximal Cohen-Macaulay. Similarly, by counting the depths of the modules in (\ref{t1}.3), we see that (ii) implies (i). Hence, by our claim and the induction hypothesis, $\depth (M\tensor_RN)\geq n$ if and only if $\Ext_R^{d}(M, N^{\vee})=0$ \emph{and} 
$\Ext_R^{i}(M_1, N^{\vee})=0$ for all $i=d-n+2,\dots, d-1,d$.  Since $M$ is a first syzygy of $M_1$, we are done. 
\end{proof}

\begin{cor}\label{cor1} Let $R$ be a $d$-dimensional local Cohen-Macaulay ring with a canonical module $\omega_R$ and let $M$ and $N$ be maximal Cohen-Macaulay $R$-modules. Assume $M$ is locally free on $U_R$ (e.g., $R$ has an isolated singularity) and $n$ is an integer such that $1\leq n\leq d$. Then $\depth(M\otimes_{R}N)\geq n$ if and only if $\Ext^i_R(M,N^{\vee})=0$ for all $i=d-n+1,\dots, d-1,d$.
\end{cor}

We will use Corollary \ref{cor1} and investigate the depth of $M\otimes_{R}M^{\ast}$. First we briefly review some of the related results in the literature. 

Auslander \cite[3.3]{Au} proved that, under mild conditions, good depth properties of $M$ and $M\otimes_{R}M^{\ast}$ force $M$ to be free. In particular the depth formula rarely holds for the pair $(M, M^{\ast})$. More precisely the following result can be deduced from the proof of \cite[3.3]{Au}:

\begin{thm}\label{Aus} (Auslander \cite{Au}, cf. also \cite[5.2]{HW1}) Let $R$ be a local Cohen-Macaulay ring and let $M$ be a finitely generated torsion-free $R$-module. Assume $M_{p}$ is a free $R_{p}$-module for each $p\in X^{1}(R)$. If $M\otimes_{R}M^{\ast}$ is reflexive, then $M$ is free.
\end{thm}

The conclusion of Auslander's result fails if $M\tensor_RM^*$ is a torsion-free module that is \emph{not} reflexive:

\begin{eg} \label{eg1}Let $R=k[[X,Y,Z]]/(XY-Z^2)$ and let $I$ be the ideal of $R$ generated by $X$ and $Y$.
Then it is clear that $R$ is a two-dimensional normal hypersurface domain and $I$ is locally free on $U_R$. Consider the short exact sequence:
$$(\ref{eg1}.1) \;\; \;0 \to I \to R \to R/I \to 0$$
Since $\dim(R/I)=0$, it follows from (\ref{eg1}.1) and the depth lemma that $\depth(I)=1$. Furthermore, as $R/I$ is torsion and $\Ext^{1}_{R}(R/I,R)=0$ \cite[3.3.10]{BH}, applying $\Hom(-,R)$ to (\ref{eg1}.1), we conlude that $I^{\ast}=\Hom(I,R) \cong R$. Therefore $I\otimes_{R}I^{\ast}\cong I$ is a torsion-free module that is not reflexive.
\end{eg}

Example \ref{eg1} raises the question of what can be deduced if one merely assumes $M\tensor_RM^*$ is torsion-free in Theorem \ref{Aus}. Auslander studied this question and proved the following result:

\begin{thm}\label{Aus2} (\cite[3.7]{Au}) Let $R$ be an \emph{even} dimensional regular local ring and let $M$ be a finitely generated $R$-module. Assume $M$ is locally free on $U_R$. Assume further that $\depth(M) =\depth(M^{\ast})$. If $\depth(M\otimes_{R}M^{\ast})>0$, then $M$ is free.  
\end{thm}

Auslander's original result assumes that the ring considered in Theorem \ref{Aus2} is unramified but this assumption can be removed by Lichtenbaum's $\Tor$-rigidity result \cite{Li}. Auslander also showed that such a result is no longer true for odd dimensional regular local rings; if $R$ is a regular local ring of odd dimension greater than one, then there exists a non-free finitely generated $R$-module $M$ such that $M$ is locally free on $U_R$, $M\cong M^{\ast}$ and $M\otimes_{R}M^{\ast}$ is torsion-free. This is fascinating since it indicates the parity of the dimension of $R$ may affect the homological properties of $R$-modules. 

Next we will prove that the conclusion of Theorem \ref{Aus2} carries over to certain types of hypersurfaces. Recall that if $R$ is a local Gorenstein ring and $M$ is a finitely generated $R$-module such that $M$ is locally free on $U_{R}$ and $\depth(M)>0$ (respectively $\depth(M)>1$), then $M$ is torsion-free (respectively reflexive). Note that the depth of the zero module is defined as $\infty$ (cf. \cite{HJW}).

\begin{prop} \label{hyp} Let $R$ be a even dimensional hypersurface with an isolated singularity such that $\widehat R\cong S/(f)$ for some unramified (or equicharacteristic) regular local ring $S$ and let $M$ be a finitely generated $R$-module. Assume $M$ is locally free on $U_R$. Assume further that $\depth(M) =\depth(M^{\ast})$. If $\depth(M\tensor_RM^{\ast})>0$, then $M$ is free.
\end{prop}

\begin{proof} Notice $R$ is a domain since it is normal. If $M^{\ast}=0$, then $\depth(M)=\depth(0)=\infty$ so that $M=0$. Hence we may assume $M^{\ast} \neq 0 \neq M$. Therefore, since $\depth(M)=\depth(M^{\ast})\geq 1$ and $M$ is locally free on $U_{R}$, $M$ is torsion-free. This shows that $M$ can be embedded in a free module: 
$$(\ref{hyp}.1)\;\;\; 0 \rightarrow M \rightarrow R^{(m)} \rightarrow M_{1} \rightarrow 0 $$
Tensoring (\ref{hyp}.1) with $M^{\ast}$, we conclude that $\Tor_{1}^{R}(M_{1},M^{\ast})=0$. It follows from \cite[4.1]{Da3} that the pair $(M_{1},M^{\ast})$ is $\Tor$-rigid. Thus $\Tor_{i}^{R}(M_{1},M^{\ast})=0=\Tor_{i}^{R}(M,M^{\ast})$ for all $i\geq 1$. Now Theorem \ref{Choi} implies that $\displaystyle{\depth(M)+\depth(M^{\ast})=\dim(R)+\depth(M\otimes_{R}M^{\ast})}$. 
Write $\dim(R)=2 \cdot n$ for some integer $n$. Then $\depth(M\tensor_RM^{\ast})=2 \cdot (\depth(M)-n)$ is a positive integer by assumption. This implies $\depth(M\otimes_{R}M^{\ast})\geq 2$ so that the result follows from Theorem \ref{Aus}.
\end{proof}

We suspect that Proposition \ref{hyp} is true for all even dimensional {\it complete intersection} rings. 

\begin{conj}\label{conjAu}
Let $R$ be a even dimensional complete intersection with an isolated singularity and let $M$ be a finitely generated $R$-module. Assume $M$ is locally free on $U_R$.  Assume further that $\depth(M) =\depth(M^{\ast})$. If $\depth(M\tensor_RM^{\ast})>0$, then $M$ is free.
\end{conj}

As an additional supporting evidence, we prove a special case of Conjecture \ref{conjAu}, namely the case where $M$ is maximal Cohen-Macaulay:
 
\begin{cor} \label{c6} Let $R$ be a even dimensional local Gorenstein ring and let $M$ be a maximal Cohen-Macaulay $R$-module that is locally free on $U_R$. Assume $\CI_{R}(M)<\infty$ and $\depth(M\otimes_{R}M^{\ast})>0$. Then $M$ is free. 
\end{cor}

\begin{proof} It follows from Corollary \ref{cor1} that $\Ext_{R}^{d}(M,M)=0$ where $d=\dim(R)$. Since $d$ is even, \cite[4.2]{AvBu} implies that $\pd_{R}(M)<\infty$. Hence $M$ is free by the Auslander-Buchsbaum formula.
\end{proof}

Recall that modules over complete intersections have finite complete intersection dimension. Therefore we immediately obtain Corollary \ref{t2} as advertised in the introduction:

\begin{cor}
Let $R$ be an even dimensional complete intersection that has an isolated singularity and let $M$ be a maximal Cohen-Macaulay $R$-module. If $\depth(M\tensor M^*)>0$, then $M$ is free.
\end{cor}

We now present examples showing that the conclusion of Corollary \ref{c6} fails for odd dimensional local rings; more precisely we show that, for each positive odd integer $n$, there exists a $n$-dimensional hypersurface $R$ and a \emph{nonfree} maximal Cohen-Macaulay $R$-module $M$ such that $R$ has an isolated singularity and $\depth(M\otimes_{R}M^{\ast})>0$. We will use the next result which  is known as \emph{Kn\"{o}rrer's periodicity theorem} (cf. \cite[12.10]{Yo}):

\begin{thm} (Kn\"{o}rrer)\label{kr}  Let $R=k[[X_{1}, X_{2}, \dots, X_{n}]]/(f)$ and let $R^{\sharp\sharp}=R[[U,V]]/(f+UV)$ where $k$ is an algebraically closed field of characteristic zero. Suppose  
$\underline{\MCM}(R)$ denotes the stable category of maximal Cohen-Macaulay $R$-modules. Then there is an equivalence of categories:
$$\Omega: \underline{\MCM}(R) \to \underline{\MCM}(R^{\sharp\sharp})$$
Here the objects of $\underline{\MCM}(R)$ are maximal Cohen-Macaulay $R$-modules and the $\Hom$ sets are defined as $\displaystyle{ \underline{\Hom}_{R}(M,N)=\frac{\Hom_{R}(M,N)}{S(M,N)} }$ where $S(M,N)$ is the $R$-submodule of $\Hom_{R}(M,N)$ that consists of $R$-homomorphisms factoring through free $R$-modules.
\end{thm}

As $\underline{\Hom}_{R}(M,N) \cong \Ext_{R}^{2}(M,N)$ and $\Omega(\syz^{R}_{1}(M)) \cong \syz^{R^{\sharp\sharp}}_{1}(\Omega(M))$ (cf. \cite[Chapter 12]{Yo}), it follows from Theorem \ref{kr} that $\Ext_{R}^{j}(M,N)=0$ if and only if $\Ext_{R^{\sharp\sharp}}^{j}(\Omega(M),\Omega(N))=0$ for $j\geq 0$. 

\begin{eg} \label{e3} We use induction and prove that if $n$ is a positive odd integer, then there exists a $n$-dimensional hypersurface $R$ and a non-free maximal Cohen-Macaulay $R$-module $M$ such that $R$ has an isolated singularity and $\depth(M\otimes_{R}M^{\ast})>0$, that is, $M\otimes_{R}M^{\ast}$ is torsion-free. Throughout the example, $k$ denotes an algebraically closed field of characteristic zero.

Assume $n=1$. Let $R=k[[X,Y]]/(f)$ for some element $f$ in $k[[X,Y]]$ such that $f$ is reducible and has no repeated factors. We pick $M=k[[x,y]]/(g)$ for some element $g\in k[[X,Y]]$ such that $g$ divides $f$. Then it can be easily checked that $R$ has an isolated singularity (that is $R$ is reduced), $M$ is torsion-free, $M^{\ast}\cong M$ and $M\otimes_{R}M^{\ast}\cong M$. It is also easy to see that $\Ext_{R}^{2i-1}(M,M)=0$ for all $i\geq 1$.  

Assume now $n=2d+1$ for some positive integer $d$. Suppose there exist a hypersurface $R=k[[X_{1}, X_{2}, \dots, X_{2d}]]/(f)$ that has an isolated singularity and a non-free maximal Cohen-Macaulay $R$-module $M$ such that $\Ext^{2i-1}_{R}(M,M)=0$ for all $i\geq 1$. Then, using the notations of Theorem \ref{kr}, we set $S=R^{\sharp\sharp}$ and $N=\Omega(M)$. Hence $S$ is a $n$-dimensional hypersurface that has an isolated singularity and $N$ is a non-free maximal Cohen-Macaulay $S$-module. It follows from the discussion following Theorem \ref{kr} that $\Ext^{2i-1}_{S}(N,N)=0$ for all $i\geq 1$. Thus Theorem \ref{t1} implies that $\depth_{S}(N\otimes_{S}N^{\ast})>0$ where $N^{\ast}=\Hom_{S}(N,S)$ (indeed, as $N$ is not free, $\depth_{S}(N\otimes_{S}N^{\ast})=1$ by Theorem \ref{Aus}.) 
\end{eg}

\begin{rmk} It is proved in Example \ref{e3} that there are many one-dimensional hypersurfaces $R$ and non-free finitely generated $R$-modules $M$ such that $M$ and $M\otimes_{R}M^{\ast}$ are torsion-free. However it is not known whether there exist such modules over one-dimensional complete intersection \emph{domains} of codimension at least two. This was first addressed by Huneke and Wiegand in \cite[page 473]{HW1} (cf. also \cite[4.1.6]{Ce} and \cite[3.1]{HW1}).
\end{rmk}

The Auslander-Reiten conjecture (for commutative local rings) states that if $R$ is a local ring and and $M$ is a finitely generated $R$-module such that $\Ext_R^i(M,M\oplus R)=0$ for all $i>0$, then $M$ is free \cite{AuRe}. Huneke and Leuschke \cite{HL} proved that this long-standing conjecture is true for Gorenstein rings that are complete intersections in codimension one (see also \cite[Theorem 3]{Ar} and \cite[5.9]{HJ}). As another application of Corollary \ref{cor1} we give a short proof of this result.

\begin{thm}(Huneke-Leuschke \cite{HJ})\label{Ar} Let $R$ be a local Gorenstein ring and let $M$ be a finitely generated $R$-module. Assume $R_{p}$ is a complete intersection for all $p\in X^{1}(R)$. If $\Ext^{i}_{R}(M,M\oplus R)=0$ for all $i>0$, then $M$ is free.
\end{thm} 

\begin{proof}[A proof of Theorem \ref{Ar}] Since $R$ is Gorenstein and $\Ext^{i}_{R}(M,R)=0$ for all $i>0$, $M$ is maximal Cohen-Macaulay. We proceed by induction on $\dim(R)$. If $R$ has dimension at most one, then the result follows from \cite[1.8]{ADS}. Otherwise, by the induction hypothesis, $M$ is locally free on $U_R$. Hence Corollary \ref{cor1} implies that $\depth(M\otimes_{R}M^{\ast})\geq 2$.  Therefore $M$ is free by Theorem \ref{Aus}.
\end{proof}

Next we exploit Theorem \ref{t1} and construct finitely generated modules $M$ and $N$ such that $M\tensor_RN$ has high depth. 

Recall that the divisor class group $\Cl(R)$ of a normal domain $R$ is the group of isomorphism classes of rank one reflexive $R$-modules. If $[I]$ represents the class of an element in $\Cl(R)$, then the group law of this group can be defined via: $[\Hom_R(I,J)]=[J]-[I]$ (cf. for example \cite{BourCommChp7}).

\begin{prop}\label{goodDepth}
Let $R$ be a $d$-dimensional local Cohen-Macaulay normal domain with a canonical module $\omega_{R}$. Assume there exists a nonfree maximal Cohen-Macaulay $R$-module $N$ of rank one.  Assume further that $N$ is locally free on $U_R$.  If $d\geq 3$, then there exists a nonfree finitely generated $R$-module $M$ such that $\depth(M\tensor N) \geq d-2$. 
\end{prop}

\begin{proof} Recall that maximal Cohen-Macaulay modules over normal domains are reflexive \cite{BH}. Therefore $[\omega_R]\in \Cl(R)$. Furthermore $N^{\vee}$ represents an element in $\Cl(R)$ (cf. for example \cite[1.5]{Vas}). Thus $[\Hom(N^*,N^{\vee})]=[\omega_R]$ and hence $\Hom(N^*,N^{\vee}) \cong \omega_R$ (Recall that $()^{\vee}=\Hom(-,w_{R})$). This implies that $\Hom(N^*,N^{\vee})$ is maximal Cohen-Macaulay. Let $M_0=N^*$. Since $M_0$ is locally free on $U_R$ and $N^{\vee}$ is maximal Cohen-macaulay, it follows from \cite[2.3]{Da4} that $\Ext_R^i(M_0, N^{\vee})=0$ for all $i=1, \dots ,d-2$. Applying the pushforward construction twice, we obtain a module $M$ such that $M_0$ is a second syzygy of $M$. Thus $\Ext_R^i(M,N^{\vee})=0$ for all $i=3, \dots, d$. Now Theorem \ref{t1} implies that the depth of $M\tensor_RN$ is at least $d-2$.
\end{proof}

\begin{rmk} \label{Veronese} Let $k$ be a field and let $S=k[[x_1,\cdots,x_d]]$ be the formal power series over $k$. For a given integer $n$ with $n>1$, we let $R=S^{(n)}$ to be the subring of $S$ generated by monomials of degree $n$, that is, the  $n$-th \emph{Veronese} subring $S$. Then $R$ is a Cohen-Macaulay local ring with a canonical module $w_{R}$ (since it is complete), $R$ has an isolated singularity. The class group of $R$ is well-understood   (cf. for example \cite[Example 4.2]{Anu} or \cite[Example 2.3.1 and 4.2.2]{BG}). We have $\Cl(R) \cong \ZZ/n\ZZ$ and it is generated by the element $L = x_1S\cap R$.
\end{rmk}

Throughout the rest of the paper, $k[[x_1,\cdots,x_d]]^{(n)}$ will denote the $n$th ($n>1$) Veronese subring of the formal power series ring $k[[x_1,\cdots,x_d]]$ over a field $k$.

\begin{eg}\label{exDepth} Let $R=k[[x_1,\cdots,x_d]]^{(n)}$. If $[N]$ is a nontrivial element in $\Cl(R)$, then it follows from Proposition \ref{goodDepth} and Remark \ref{Veronese} that there exists a nonfree finitely generated $R$-module $M$ such that $\depth(M\tensor_RN)\geq d-2$. 
\end{eg}

We will finish this section with an application concerning semidualizing modules. Recall that a finitely generated module $C$ over a Noetherian ring $R$ is called \textit{semidualizing} if the natural homothety homomorphism $R \longrightarrow \Hom_{R}(C,C)$ is an isomorphism and $\Ext^{i}_{R}(C,C)=0$ for all $i\geq 1$ (cf. for example \cite{Wag1} for the basic properties of semidualizing modules.). We will prove that there is no nontrivial semidualizing module over $R=k[[x_1,\cdots,x_d]]^{(n)}$, that is, if $C$ is a semidualizing module over $R$, then either $C \cong R$ or $C \cong w_{R}$, where $w_{R}$ is the canonical (dualizing) module of $R$ (cf. Corollary \ref{last corollary}). We first start proving a lemma:

\begin{lem} \label{lemma for semidualizing} Let $R=k[[x_1,\cdots,x_d]]^{(n)}$ and let $L$ represents the generator of $\Cl(R)=\ZZ/n\ZZ$ as in remark \ref{Veronese}. Then $\displaystyle{\mu(L^{(i)})=\frac{(n+d-i-1)!}{(d-1)!}}$ for $i=1, \dots, n-1$, where $\mu(L^{i})$ denotes the minimal number of generators of $L^{(i)}$, the $i$-symbolic power of $L$, which represents the element $i[L]$ in $\Cl(R)$.
\end{lem}

\begin{proof} Note that $S=R\oplus L \oplus L^{(2)} \oplus \dots \oplus L^{(n-1)}$, where $L^{(i)}$ is  generated by the monomials of degree $n$ divisible by $x^{i}_{1}$. Therefore $\mu(L^{(i)})$ is the number of monomials of degree $n-i$, which is exactly $\displaystyle{\frac{(n-i+d-1)!}{(d-1)!}}$. 
\end{proof}

\begin{prop} \label{proposition for semidualizing} Let $R=k[[x_1,\cdots,x_d]]^{(n)}$ and let $I$ and $J$ be elements in $\Cl(R)$. If $I\otimes_{R}J$ is reflexive, then $I\cong R$ or $J\cong R$.
\end{prop}

\begin{proof} As $I\otimes_{R}J$ is reflexive, it is torsion-free so that $I\otimes_{R}J \cong IJ$. This implies $\mu(I)\mu(J)=\mu(IJ)$. Assume $L$ is a generator of $\Cl(R)=\ZZ/n\ZZ$ so that $I=L^{(i)}$ and $J=L^{(j)}$ for some $i$ and $j$. Suppose $I$ and $J$ are not free, that is, $0<i\leq n-1$ and $0<j\leq n-1$. 

Supose first $i+j<n$. Notice, if $a<b<n$, then $\mu(L^{(a)})>\mu(L^{(b)})$ by Lemma \ref{lemma for semidualizing}. Therefore the case where $i+j<n$ contradicts the fact that $\mu(I)\mu(J)=\mu(IJ)$. Next assume $i+j\geq n$, that is, $i+j=n+h$ for some nonnegative integer $h$. Then $L^{(i+j)} \cong L^{(h)}$ and hence, by Lemma \ref{lemma for semidualizing}, we obtain: 
$$(n-i+d-1)! \cdot (n-j+d-1)! = (n-h+d-1)! \cdot (d-1)!$$
Setting $N=n-h+2(d-1)$, we conclude:
$$
{N\choose n-i+d-1} = {N\choose n-h+d-1}
$$
Now, without loss of generality, we may assume $i\leq j$. Set $N_{1}=n-i+d-1$ and $N_{2}=n-h+d-1$. Then $\displaystyle{N_{1}>N_{2}\geq \frac{N}{2}}$. Hence the above equality is impossible since binomials are unimodular. Therefore either $i=0$ or $j=0$, that is, either $I\cong R$ or $J\cong R$.
\end{proof}

\begin{cor} \label{corollary for semidualizing} Let $R=k[[x_1,\cdots,x_d]]^{(n)}$ and let $I$ and $J$ be elements in $\Cl(R)$. If $\Ext^{d-1}_{R}(I,J)$ $=\Ext_{R}^{d}(I,J)=0$, then $I \cong R$ or $J \cong w_{R}$.
\end{cor}

\begin{proof} We know, by Remark \ref{Veronese}, that $I$ and $J$ are maximal Cohen-Macaulay $R$-modules both of which are locally free on $U_{R}$. Therefore it follows from Corollary \ref{cor1} that $\depth(I\otimes_{R}J^{\vee})>1$, that is, $I\otimes_{R}J^{\vee}$ is reflexive. Now Proposition \ref{proposition for semidualizing} implies either $I\cong R$ or $J^{\vee}\cong R$. Since $J^{\vee\vee} \cong J$, the desired conclusion follows.
\end{proof}


Sather-Wagstaff \cite[3.4]{Wag1} exhibited a natural inclusion from the set of isomorphism classes of semidualizing R-modules to the divisor class group of a normal domain $R$; every semidualizing module $C$ is a rank one reflexive module so that it represents an element in the class group $\Cl(R)$. Therefore, since $\Ext^{i}_{R}(C,C)=0$ for all $i>0$, it follows
from Corollary \ref{corollary for semidualizing} that:

\begin{cor} \label{last corollary} If $C$ is a semidualizing module over the Veronese subring $R=k[[x_1,\cdots,x_d]]^{(n)}$, then $C \cong R$ or $C\cong w_{R}$.
\end{cor}

\bibliography{a}
\bibliographystyle{plain}

\end{document}